\documentclass[a4paper,10pt]{article}
\usepackage[utf8]{inputenc}

\usepackage{amsmath}
\usepackage{amsfonts}
\usepackage{amssymb}
\usepackage{amsthm}
\usepackage{graphicx}
\usepackage{amscd}

\title{Random-Walks on Graphs and Approximation of $l_2$-Invariants}
\date{}
\author{Zenas Wong, Andrew Kricker}

\begin{document}

\maketitle

\newcommand{\C}{\mathbb{C}}
\newcommand{\norm}[1]{\lVert #1 \rVert}
\newcommand{\bol}[1]{\textbf{#1}}
\newcommand{\vonNeumannAlg}[1]{\mathcal{N}(#1)}
\newcommand{\inp}[2]{\langle #1, #2 \rangle}
\newcommand{\N}{\mathbb{N}}
\newcommand{\vnd}[1]{\dim_{\vonNeumannAlg{G}}(#1)}
\newcommand{\R}{\mathbb{R}}
\newcommand{\Z}{\mathbb{Z}}
\newcommand{\mac}[1]{\mathcal{#1}}
\newcommand{\mab}[1]{\mathbf{#1}}
\newcommand{\bracket}[1]{\langle #1 \rangle}
\newcommand{\vontrace}[2]{trace_{\vonNeumannAlg{#1}}(#2)}
\newcommand{\vondim}[2]{\dim_{\vonNeumannAlg{#1}}(#2)}
\newcommand{\Q}{\mathbb{Q}}
\newcommand{\F}{\mathbb{F}}
\newcommand{\T}{\mathbb{T}}
\newcommand{\convol}[1]{\overline{\mac{LCMG}(#1)}}
\newcommand{\contrace}[2]{trace_{\convol{#1}}(#2)}
\newcommand{\condim}[2]{\dim_{\convol{#1}}(#2)}
\newtheorem{theorem}{Theorem}
\newtheorem{lemma}{Lemma}
\newtheorem{coroll}{Corollary}
\newtheorem{propos}{Proposition}
\newtheorem{example}{Example}
\newtheorem{conjecture}{Conjecture}
\newtheorem{define}{Definition}
\newtheorem{remark}{Remark}
\newtheorem*{setting}{Setting}

\begin{abstract}
Right multiplication operators $R_w: l_2G \rightarrow l_2G$, $w \in \C[G]$, are interpreted as random-walk operators on labelled graphs that are analogous to Cayley graphs. Applying a generalization of the graph convergence defined by R. Grigorchuk and A. \.{Z}uk \cite{Grigorchuk_Zuk_1} gives a new proof and interpretation of a special case of W. L\"{u}ck's famous Theorem on the Approximation of $l_2$-Betti numbers for countable residually finite groups. In particular, using this interpretation, the proof follows quickly from standard theorems about the weak convergence of probability measures that are characterized by their moments.

% In particular, a geometric group theoretic proof is given to prove: given a descending sequence of finite index normal subgroups $\{K_n\}_{n \in \N}$ of $G$ with trivial intersection, if $w$ is an element of $\R[G]$, then the spectral density function $F(R_w)$ of the right multiplication operator $R_w$ on $l_2G$ is the pointwise a.e. limit of the sequence of spectral density functions $\{F(R_{w_n})\}_{n \in \N}$ of right multiplication operators $R_{w_n}$, where $w_n \in \R[G/K_n]$.
\end{abstract}

\section{Introduction}

A key area of interest in the theory of $l_2$-invariants is the study of the spectral density function $F(R_w): [0,\infty) \rightarrow [0,\infty)$ of a $G$-equivariant bounded operator $R_w: l_2G \rightarrow l_2G$, given by right multiplication of $w$ (an element of the complex group ring $\C[G]$) on the Hilbert space $l_2G$ of square summable formal sums: $$l_2G = \{\sum_{g \in G} \lambda_g g : \lambda_g \in \C, \sum_{g \in G} |\lambda_g|^2 < \infty \}.$$ $F(R_w)$ is used to define $l_2$-invariants of $R_w$, such as the $l_2$-Betti number, Novikov-Shubin invariant, and Fuglede-Kadison determinant.\bigskip

In general, $F(R_w)$ is difficult to compute, and there are few instances of actual computations of the $l_2$-invariants of right multiplication operators. For example, for any $w \in \C[G]$ with $G = \Z^d$, recent work by W. L\"{u}ck \cite{WLuck_SDF_Estimates_FAG} provided bounds on the spectral density function of $R_w$, and established that the Novikov-Shubin invariant and Fuglede-Kadison determinant of $R_w$ are both positive. In \cite{Grabowski1}, L. Grabowski discovered particular groups $G$ for which certain elements $w \in \C[G]$ are such that the $l_2$-Betti number of $R_w$ is irrational.\bigskip

This paper studies the spectral density function of $R_w$ using the theory of random walk (or Markov) operators of groups. This relies on  \cite{Grigorchuk_Zuk_1}, where R. Grigorchuk and A. \.{Z}uk defined convergence of sequences of locally finite marked graphs, and showed that the spectral measure associated to the random-walk (or Markov) operator of such a graph is continuous with respect to this type of graph convergence. This was subsequently used in \cite{Grigorchuk_Zuk_2} to explicitly compute the spectrum of a Markov operator of the lamplighter group, and the approach was further extended in \cite{KambitesPedroSilvaSteinberg} to Markov operators of groups generated by Cayley automata.\bigskip

An important connection between these 2 theories is that Markov operators of a finitely generated group $G$ set up the following Correspondence $(\star)$ between:

\begin{enumerate}
 \item Random walks on the Cayley graph $\Gamma_{G,S}$ of $G$ with respect to a finite symmetric generating set $S$, and
 \item $G$-equivariant bounded operators $R_m: l_2G \rightarrow l_2G$, given by right multiplication of the self-adjoint element $m = \frac{1}{|S|} \sum_{s \in S} s \in \C[G]$. $R_m$ is known as the random-walk operator (or Markov operator) associated to the Cayley graph $\Gamma_{G,S}$.
\end{enumerate}

In particular, computing the spectral measures of random walk operators on the Cayley graph $\Gamma_{G,S}$ gives information on the spectral density function of $R_m$ and the associated $l_2$-invariants. For example, the main result of \cite{Grigorchuk_Zuk_2} was used to give a counter-example \cite{Grigorchuk_Zuk_Linnell_Schick} to the Strong Atiyah Conjecture concerning the range of values of $l_2$-Betti numbers.\bigskip

The aim of this paper is to extend this correspondence to self-adjoint elements $z = w w^\ast \in \C[G]$ for an arbitrary element $w$ of $\C[G]$, and use it to study $F(R_w)$. The approach is to introduce a graph $\Gamma$ that is analogous to Cayley graphs, such that right multiplication by $w$ is the random walk operator associated to $\Gamma$. In keeping with Correspondence $(\star)$, $R_w$ will be called the Markov-type operator associated to $\Gamma$. When $G$ is a finitely generated residually finite group, this approach yields a new proof for the pointwise almost-everywhere convergence to $F(R_w)$ by means of finite-dimensional analogues. To explain this, one will fix the following notations:

\begin{setting}
\label{MainSetting}
Let $G$ be a finitely generated residually finite group. Suppose that $G$ has a nested sequence of finite index normal subgroups: $$\{K_n \triangleleft G : [G:K_n] < \infty, \  K_n \supseteq K_{n+1}\}_{n \in \N},$$ such that $\bigcap_{n \in \N} K_n = \{1\}$. For $A \in M_{a \times b}(\C[G])$, let $R_A: (l_2G)^a \rightarrow (l_2G)^b$ be the operator given by right multiplication by $A$. Set: $$A_n = \pi_n(A) \in M_{a \times b}(\C[G/K_n]),$$ obtained by applying the canonical projection $\pi_n: \C[G] \rightarrow \C[G/K_n]$ to the entries of $A$, and let $R_{A_n}: l_2(G/K_n)^a \rightarrow l_2(G/K_n)^b$ be the associated right multiplication operator. 
\end{setting}

% Let $\beta^{(2)}(T_n)$ be the corresponding $l_2$-invariant of $T_n$ that is analogous to $\beta^{(2)}(T)$; then one would like to know if: \[\beta^{(2)}(T) = \lim_{n \rightarrow \infty} \beta^{(2)}(T_n)\]

% An important conjecture concerns whether the $l_2$-Betti number of $R_A$ can be approximated in terms of the sequence of $l_2$-Betti numbers of the $T_n$'s:

The main result is a new proof of the following theorem:

\begin{theorem}[Pointwise a.e. Convergence of Spectral Density Functions]
\label{PointwiseConvergence_SDF_ResiduallyFiniteGroup}
Let $G$ be a countable residually finite group as described in the Setting, and let $A$ and $A_n$ be given by $w \in \C[G]$ and $w_n =  \pi_n(w)$ respectively (i.e. when $a = b = 1$). Let $F(R_w)(\lambda)$ and $F(R_{w_n})(\lambda)$ be the spectral density functions of $R_w$ and $R_{w_n}$ respectively. Then for all $\lambda \geq 0$ at which $F(R_w)$ is continuous, one has that $$\lim_{n \rightarrow \infty} F(R_{w_n})(\lambda) = F(R_w)(\lambda).$$
\end{theorem}

A corollary of Theorem \ref{PointwiseConvergence_SDF_ResiduallyFiniteGroup} is a special case (for $a = b = 1$) of a famous theorem of W. L\"{u}ck \cite{WLuck_Approximation_FDAnalogues,WLuck_Approximation_Survey} concerning the approximation of the $l_2$-Betti number $b^{(2)}(R_A)$ by the sequence of $l_2$-Betti numbers $b^{(2)}(R_{A_n})$ arising from the finite dimensional operators $R_{A_n}$.

\begin{theorem}[W. L\"{u}ck]
\label{ApproximationConjecture_L2BettiNumbers}
Let $G$ be as described in the Setting. For $A \in M_{a \times b}(\Q[G])$, let $b^{(2)}(R_A)$ and $b^{(2)}(R_{A_n})$ be the $l_2$-Betti numbers of $R_A$ and $R_{A_n}$ respectively. Then: \[b^{(2)}(R_A) = \lim_{n \rightarrow \infty} b^{(2)}(R_{A_n}).\]

% Then $G$ satisfies each of the following equivalent conditions:

% \begin{enumerate}
%  \item Let $A \in M_{a \times b}(\Q[G])$, and $b^{(2)}(T)$ and $b^{(2)}(T_n)$ the $l_2$-Betti numbers of $T = R_A$ and $T_n = R_{A_n}$ respectively. Then: \[b^{(2)}(T) = \lim_{n \rightarrow \infty} b^{(2)}(T)\]
%  \item Let $X$ be a $G$-CW complex of finite type; then $X[n] = K_n \backslash X_i$ is a $G/K_n$-CW comeplex of finite type and for all $p \geq 0$, there is the following limit for the $p$th Betti numbers: \[b_p^{(2)}(X) = \lim_{n \rightarrow \infty} b_p^{(2)}(X[i])\]
% \end{enumerate}

\end{theorem}

Note Theorem \ref{PointwiseConvergence_SDF_ResiduallyFiniteGroup} has already been proved by W. L\"{u}ck in Theorem 16.3 of \cite{WLuck_Approximation_Survey} by means of a functional analytic proof. While the techniques behind the new proof of Theorem \ref{PointwiseConvergence_SDF_ResiduallyFiniteGroup} constitute well-known results, they are nonetheless part of an important and increasingly useful approach of applying tools from geometric group theory to give insights into conjectures involving spectral density functions of arbitrary right multiplication operators and their associated $l_2$-invariants. Examples of the usefulness of this approach can be seen in \cite{RSauer_Thesis,Grabowski2}.\bigskip

This paper grew out of our attempts to understand the methods to compute $l_2$-Alexander invariant for knots $K$, which are such that the fundamental group $\pi_1(\mathbb{S}^3 \backslash K)$ (of the complement of $K$ in the 3-sphere $\mathbb{S}^3$) is a countable residually finite group. This is an $l_2$-invariant analogue of twisted Alexander polynomials that was introduced by W. Li and W. Zhang in \cite{WLi_WZhang}. In particular, knots such as 2-bridge knots $K$ are such that $G = \pi_1(\mathbb{S}^3 \backslash K)$ admits a deficiency 1 presentation, and therefore the $l_2$-Alexander invariant of $K$ can, in principle, be computed from an element $w$ of $\C[G]$ (although for most cases, direct computations are difficult to perform). The $l_2$-Alexander invariant is an example of an $l_2$-Alexander torsion; $l_2$-Alexander torsions were studied by J. Dubois, S. Friedl, and W. L\"{u}ck in \cite{WLuck_SFriedl_JDubois_1, WLuck_SFriedl_JDubois_2}.\bigskip

The paper is organized as follows. Section \ref{Preliminaries} presents relevant theory on $l_2$-invariants of right multiplication operators, and the weak convergence of measures. Section \ref{lcmgTheory} presents the main geometric group theoretic tool of labelled connected marked graphs (denoted \emph{lcmg}'s) and the corresponding concept of convergence. Section \ref{MainResult_Proof} gives the proof of the main result.

\section{Preliminaries}
\label{Preliminaries}

\subsection{$l_2$-Invariants}

For details on the theory of $l_2$-invariants, see \cite{WLuck_Monograph}, Chapters 1, 2, and 13. Let $G$ be a discrete group. The group ring $\C[G]$ comes with an involution $\ast$, such that for $\sum_{g \in G} \lambda_g g \in \C[G]$, $\lambda_g \in \C$ , $(\sum_{g \in G} \lambda_g g)^\ast = \sum_{g \in G} \overline{\lambda_g} g^{-1}$, where $\overline{\lambda}$ is the complex conjugate of $\lambda$. One says that $w \in \C[G]$ is \emph{self-adjoint} if $w^\ast = w$.\bigskip

Let $w \in \C[G]$, and let $R_w: l_2G \rightarrow l_2G$ be the associated right multiplication operator. The adjoint of $R_w$ is $$(R_w)^\ast = R_{w^\ast}: l_2G \rightarrow l_2G,$$ which is also a right multiplication operator. $R_w$ is self-adjoint (i.e. $(R_w)^\ast = R_w$) if and only if $w$ is self-adjoint. $R_w$ is itself an element of the group von Neumann algebra $\vonNeumannAlg{G}$ consisting of all $G$-equivariant bounded linear operators $l_2G \rightarrow l_2G$; $\vonNeumannAlg{G}$ has a finite faithful normal trace: \[\vontrace{G}{T} = \inp{T(1)}{1} \quad , \ T \in \vonNeumannAlg{G},\] where $1 \in l_2G$ is the identity element of $G$.\bigskip

For an arbitrary $T \in \vonNeumannAlg{G}$, $T^\ast \circ T$ is self-adjoint, and therefore has a family of spectral projections: $$\{E^{T^\ast \circ T}_B = \chi_B(T^\ast \circ T)\}_B,$$ where the collection is taken over all Borel subsets $B$ of $[0,\norm{T}^2]$, and $\chi_B$ is the characteristic function of $B$.\bigskip

The \emph{spectral density function} $F(T)(\lambda)$ is given by: \[F(T)(\lambda) = \vontrace{G}{E^{T^\ast \circ T}_{[0,\lambda^2]}}.\] This is a monotonically increasing right continuous function. The \emph{$l_2$-Betti number} of $T$ is defined to be $b^{(2)}(T) := F(T)(0)$. Roughly, $b^{(2)}(T)$ is a measure of the size of the kernel of $T$; one also has that $b^{(2)}(T) = 0$ if and only if $T$ is injective.

\subsection{Weak Convergence of Probability Measures}

For details, see Chapter 13 of \cite{Klenke}, Section 11.4 of \cite{Rosenthal}, and also \cite{Feller}. Let $\{\mu_n\}$ be a sequence of probability measures on a metric space $X$ (with Borel $\sigma$-algebra $\mac{B}(X)$). One says that $\{\mu_n\}$ converges weakly to a probability measure $\mu$ if $\lim_{n \rightarrow \infty} \mu_n(B) = \mu(B)$ for all $B \in \mac{B}(X)$ such that $\mu(\partial B) = 0$, where $\partial B$ is the boundary of $B$. The limit measure of weak convergence can be shown to be unique.\bigskip

Now for $X = \R$, let $\mu: \mac{B}(\R) \rightarrow \R_{\geq 0}$ be a probability measure on $\R$. The \emph{cumulative distribution function} (CDF) $F_\mu: \R \rightarrow \R_{\geq 0}$ associated to $\mu$ is $F_\mu(x) := \mu(-\infty,x]$. Note that the intervals $(-\infty,x]$ have measure zero boundary.

\begin{propos}[Theorem 13.23, \cite{Klenke}]
\label{WeakConvergenceMeasures_CDFPointwiseConvergence_Equiv}
For probability measures $\{\mu_n\}_{n \in \N}$ and $\mu$ on $\mac{B}(\R)$, the following are equivalent:

\begin{enumerate}
 \item $\{\mu_n\}$ converges weakly to $\mu$.
 \item For all $x$ such that $F_\mu$ is continuous, $\lim_{n \rightarrow \infty} F_{\mu_n}(x) = F_\mu(x)$.
\end{enumerate}

\end{propos}

Let $\mac{P}(\R)$ be the set of Borel probability measures on $\R$ such that all polynomials with real coefficients (denoted $\R[x]$) belong to $L^1(\mu)$; define an equivalence relation on $\mac{P}(\R)$: $\mu_1 \sim \mu_2$ if and only if for every $p \in \R[x]$, $\int_\R p \ d\mu_1 = \int_\R p \ d\mu_2$. For $\mu \in \mac{P}(\R)$ and a measurable function $f$ on $\R$, the \emph{expected value} of $f$ is $E_\mu[f] = \int_\R f(x) d\mu(x)$. For $k \in \N$, the $k$th moment of $\mu$ is $\mu^{(k)} = E_\mu[x^k]$. $\mu$ is said to be \emph{characterized by its moments} if the equivalence class of $\mu$ consists of a single measure. By the Stone-Weierstrass Theorem, any compactly supported measure is characterized by its moments.\bigskip

The following result is a classical result in the theory of probability measures.

\begin{theorem}[Theorem 11.4.1, \cite{Rosenthal}]
\label{MomentsConvergence_CharacterizedBy_Moments_WeakConvergence}
Suppose that $\{\mu_n\}_{n \in \N} \subseteq \mac{P}(\R)$ is a sequence of measures such that:

\begin{enumerate}
 \item For every $p \in \R[x]$, $\lim_{n \rightarrow \infty} \int_\R p \ d\mu_n = \int_\R p \ d\mu$.
 \item $\mu$ is characterized by its moments.
\end{enumerate}

Then $\{\mu_n\}_{n \in \N}$ converges weakly to $\mu$.

\end{theorem}

\section{Labelled Connected Marked Graphs}
\label{lcmgTheory}
This section develops the notion of convergence of labelled connected marked graphs (denoted lcmg) $(X,v)$, where $v$ is a distinguished vertex of $X$. This is a slight extension of the convergence of connected marked graphs as described in \cite{Grigorchuk_Zuk_1}. The set of all lcmg's forms a metric space, and will be denoted by $\mac{LCMG}$. For each lcmg $(X,v)$, one obtains an operator $M_{X,v}$ that describes a random walk on the lcmg; $M_{X,v}$ will be called the Markov-type operator associated to $X$. To $M_{X,v}$ will be associated a probability measure $\mu^X_v$, which will be called the Kesten spectral measure of $X$ at $v$. These spectral measures generalize the Kesten spectral measures that occur in \cite{Grigorchuk_Zuk_3}, which are defined for Markov operators of Cayley graphs; see also \cite{KambitesPedroSilvaSteinberg}. The key result is that the map $(X,v) \mapsto M_{X,v} \mapsto \mu^X_v$ is weakly continuous with respect to convergence of lcmg's.\bigskip

A \emph{labelled connected marked graph} (lcmg) is a connected directed graph $X = (V,E)$ with a distinguished (or marked) vertex $v \in V(X)$, such that:

\begin{enumerate}
 \item $V(X)$ and $E(X)$ are both countable.
 \item \label{FiniteInOutDegree_Per_Vertex} Each vertex has finite in and out degrees.
 \item Edges of $X$ carry labels in $\C$, which are determined by a map $L_X: E(X) \rightarrow \C$. Edges are allowed to have label $0 \in \C$.
 \item \label{AtMostOneSelfLoop_Per_Vertex} For each vertex, at most one self-loop is allowed at each vertex.
 \item \label{NoMultipleDirectedEdges}For any 2 vertices $v_1, v_2 \in V$, there are no multiple directed edges from $v_1$ to $v_2$; the unique directed edge $v_1 \rightarrow v_2$ will be denoted by $(v_1,v_2)$.
\end{enumerate}

By ignoring the $\C$-labels on edges and taking the underlying graph, there is a metric $d_{X}$ on $V$, where $d_X(v_1,v_2)$ is given by the minimum length of a directed path in $X$ from $v_1$ to $v_2$, where each edge is assigned length 1.\bigskip

Given lcmg's $(X_1,v_1)$ and $(X_2,v_2)$, say that $(X_1,v_1)$ is lcmg \emph{isomorphic} to $(X_2,v_2)$ if there exists a bijection $f: V(X_1) \rightarrow V(X_2)$ such that:

\begin{enumerate}
\item Basepoint-preserving: $f(v_1) = v_2$.
\item Edge-preserving: $(u,v) \in E(X_1)$ if and only if  $(f(u),f(v)) \in E(X_2)$.
\item Label-preserving: $L_{X_1}( \ (u,v) \ )  = L_{X_2}( \ (f(u),f(v)) \ )$.
\end{enumerate}

Let $\mac{LCMG}$ denote the set of all lcmg's. For every $(X,v) \in \mac{LCMG}$, let $B(X,v,r)$ be the ball of radius $r > 0$ in $X$ centred at $v$; this is precisely the full subgraph of $(X,v)$ with vertices $u \in V(X)$ such that $d_X(u,v) \leq r$. Then define $\mac{D}: \mac{LCMG} \times \mac{LCMG} \rightarrow \R_{\geq 0}$, by: \begin{align*}
&\mac{D}((X_1,v_1), (X_2,v_2)) = \\
&\inf_{n \in \N} \{\frac{1}{n + 1} : \text{$B(X_1,v_1,n)$ and $B(X_2,v_2,n)$ are lcmg isomorphic}\} .
\end{align*} 

% Since $0$ is the sole limit point of $\{\frac{1}{n + 1}\}_{n \in \N}$, if $\mac{D}((X_1,v_1), (X_2,v_2)) > 0$, then $\mac{D}((X_1,v_1), (X_2,v_2)) = \frac{1}{k_{1,2} + 1}$, where $k_{1,2} \in \N$ is the largest possible number such that $B(X_1,v_1,n)$ and $B(X_2,v_2,n)$ are lcmg isomorphic. 

It is then easy to see that $\mac{D}$ is a metric on $\mac{LCMG}$. Thus, one says that a sequence of lcmg's $\{(X_n,v_n)\}$ \emph{converges} to a lcmg $(X,v)$ if $$\lim_{n \rightarrow \infty} \mac{D}((X_n,v_n),(X,v)) = 0 .$$ Up to lcmg isomorphism, the limit graph $(X,v)$ is unique.\bigskip

% , and there is also the following equivalent characterization of lcmg convergence.

% \begin{propos}
% $\mac{D}$ is a metric on $\mac{LCMG}$.
% \end{propos}
% 
% \begin{proof}
% Reflexivity and Symmetry are easy to show, so one deals with Transitivity. Suppose that there are $S$-lcmg's $(X_i,v_i)$, for $i \in \{1,2,3\}$ such that: \[\mac{D}((X_1,v_1),(X_3,v_3)) > \mac{D}((X_1,v_1),(X_2,v_2)) + \mac{D}_S((X_2,v_2),(X_3,v_3))\] where $\mac{D}((X_i,v_i), (X_j,v_j)) =  1/(k_{i,j} + 1) > 0$ for all $i$ and $j$. This implies that: \[\frac{1}{k_{1,3} + 1} > \frac{1}{k_{1,2} + 1} + \frac{1}{k_{2,3} + 1}\] and therefore $k_{1,3} < \min\{k_{1,2},k_{2,3}\}$. Assume WLOG that $k_{1,2} \leq k_{2,3}$; this means that $B(X_1,v_1,k_{1,2})$, $B(X_2,v_2,k_{1,2})$, and $B_{X_3,v_3,k_{1,2}}$ are all lcmg isomorphic, which contradicts the maximality of $k_{1,3}$.
% \end{proof}

The following lemma follows immediately from the definitions.

\begin{lemma}
\label{lcmg_Convergence_EquivalentStatements}
Let $\{(X_n,v_n)\}$ be a sequence of lcmg's. The following are equivalent:

\begin{enumerate}
 \item $\lim_{n \rightarrow \infty} \mac{D}((X_n,v_n),(X,v)) = 0$.
 \item For every $r > 0$, there exists $N \in \N$ such that if $n \geq N$, then $B(X_n,v_n,r)$ is lcmg isomorphic to $B(X,v,r)$.
\end{enumerate}

\end{lemma}

% \begin{proof}
% It is enough to assume throughout that $\mac{D}((X_n,v_n), (X,v)) > 0$ for all $n \in \N$. Since $0$ is the only limit point of the sequence $\{\frac{1}{n + 1} : n \in \N\}$, this means that $\mac{D}((X_n,v_n), (X,v)) = \frac{1}{m_n + 1}$, where $m_n \in \N$ is maximal with the property that $B(X_n,v_n,m_n)$ is lcmg isomorphic to $B(X,v,m)$.\bigskip
% 
% $\rightarrow$; Let $r > 0$ be given. Since $\lim_{n \rightarrow \infty} \mac{D}((X_n,v_n),(X,v)) = 0$, there exists $N \in \N$ such that $\frac{1}{m_n + 1} < \frac{1}{r + 1} \Longleftrightarrow r < m_n$, and therefore $B(X_n,v_n,r)$ is lcmg isomorphic to $B(X,v,r)$.\bigskip
% 
% $\leftarrow$; let $\epsilon > 0$ be given. Then for $r = \frac{1}{\epsilon} - 1 > 0$, there exists $N \in \N$ such that if $n \geq N$, then $B(X_n,v_n,r)$ is lcmg isomorphic to $B(X,v,r)$; hence by definition of $m_n$, $$n \geq N \implies m_n \geq \frac{1}{\epsilon} - 1 \Longleftrightarrow \frac{1}{m_n + 1} < \epsilon$$
% \end{proof}

% \begin{example}
% Consider any lcmg $(X,v)$; the sequence of $n$-balls $B(X,v,n)$ of $(X,v)$ converges to $(X,v)$.
% \end{example}

It is also possible to define an involution $\ast$ on $\mac{LCMG}$. For any lcmg $(X,v)$, define $(X,v)^\ast := (X^\ast,v)$ to have basepoint $v$, and for each edge $(v_1,v_2) \in E(X)$ with label $L_X( \ (v_1,v_2) \ )$, replace it with an edge $(v_2,v_1)$ with label $\overline{L_X( \ (v_1,v_2) \ )}$. Since $X$ is connected, $(X^\ast,v)$ is also a lcmg. $(X,v)$ is said to be \emph{self-involutive} if $(X,v)^\ast = (X,v)$.\bigskip

For $(X,v) \in \mac{LCMG}$, define the \emph{Markov-type operator associated to $(X,v)$}, $M_{X,v}: l_2(V(X)) \rightarrow l_2(V(X))$, as follows: \[M_{X,v}(f)(u) = \sum_{(w,u) \in E(X)} L_X( \ (w,u) \ ) \cdot f(w) \quad , \ f \in l_2(V(X)), \ u \in V(X).\] Note that $M_{X,v}$ is a well-defined linear operator due to conditions \ref{FiniteInOutDegree_Per_Vertex}, \ref{AtMostOneSelfLoop_Per_Vertex} and \ref{NoMultipleDirectedEdges} in the definition of a lcmg.\bigskip 

The inner product on $l_2(V(X))$ is given by: \[\inp{f_1}{f_2} = \sum_{x \in V(X)} f_1(x) \overline{f_2(x)}.\] Taking the standard orthonormal basis $\{\delta_x \in l_2 (V(X)) : x \in V(X)\}$, where $\delta_x(u) = 1$ when $u = x$ and is 0 otherwise: $$\inp{M_{X,v}(\delta_x)}{\delta_y} = \begin{cases}L_X( \ (x,y) \ ) \quad , \ (x,y) \in E(X) \\ 0 \quad , \ o.w.\end{cases} .$$ In general, $M_{X,v}$ may not be a self-adjoint bounded operator, but will be for special choices of lcmg's, as the next result shows.

\begin{propos}
\label{MarkovTypeOper_Speciallcmg_SelfAdjoint}
If $(X,v) \in \mac{LCMG}$ is self-involutive, then $M_{X,v}$ is a self-adjoint bounded linear operator.
\end{propos}

\begin{proof}
First check that $M_{X,v}^\ast = M_{X,v}$ on the $\C$-basis $\{\delta_x\}$:

\begin{align*}
\inp{\delta_x}{M_{X,v}(\delta_y)} = \sum_{v \in V(X)} \delta_x(v) \overline{M_{X,v}(\delta_y)(v)}\\
= \overline{M_{X,v}(\delta_y)(x)} = M_{X,v}(\delta_x)(y) = \inp{M_{X,v}(\delta_x)}{\delta_y} .
\end{align*}

Therefore $M_{X,v}$ is symmetric, but since it is also everywhere-defined, it follows from the Hellinger-Toeplitz theorem (\cite{ReedSimon}, Section 3.5) that $M_{X,v}$ is self-adjoint and bounded.

\end{proof}

So consider a self-involutive lcmg $(X,v)$, with associated Markov-type operator $M = M_{X,v}$. Let $\{E^{M}_B\}_B$ be the family of spectral projections determined by $M$, where $\norm{M} < \infty$ is the operator norm of $M$, and $B$ is a Borel subset of $[-\norm{M},\norm{M}]$. For $x,y \in V(X)$, one obtains a collection of measures on the Borel $\sigma$-algebra of $[-\norm{M},\norm{M}]$: \[\mu^X_{x,y}(B) = \inp{E^M_B(\delta_x)}{\delta_y}.\] The probability measures: $$\{\mu_x^X = \mu^X_{x,x} : x \in V(X)\}$$ will be called the \emph{Kesten spectral measures} of $X$; individually, $\mu^X_x$ will be called the Kesten spectral measure of $X$ at $x$. For the case of the Markov type operator $M$, denote $E_\lambda^M = E^M_{(-\infty,\lambda]}$. The moments of the Kesten spectral measures are given by: \begin{align*}
(\mu_a^X)^{(k)} = \int_{-\norm{M}}^{\norm{M}} \lambda^k d\mu_a^X(\lambda) = \int_{-\norm{M}}^{\norm{M}} \lambda^k d\inp{E_\lambda^M(\delta_a)}{\delta_a}\\
= \inp{M^k(\delta_a)}{\delta_a}
\end{align*}

This leads to the following result that generalizes Lemma 4 of \cite{Grigorchuk_Zuk_1}.

\begin{theorem}
\label{lcmgConvergence_Implies_KSMConvergence}
Let $\{(X_n,v_n)\}_{n \in \N}$ be a sequence of lcmg's that converges to a lcmg $(X,v)$. Then the sequence of Kesten spectral measures $\{\mu_{v_n}^{X_n}\}_{n \in \N}$  converges weakly to $\{\mu_v^X\}$.
\end{theorem}

\begin{proof}
By the computation above, the $k$th moment of $\mu_v^X$ is given by the sum, over all possible directed walks $\omega$ in $X$, starting at $v$ and of metric length $d_X(x,v) = k$, of the product of the $\C$-labels of the edges along each such $\omega$; by a walk, it is meant that any directed edge may appear more than once in $\omega$. But since the sequence $\{(X_n,v_n)\}$ converges to $(X,v)$, for $n$ sufficiently large, $B(X_n,v_n,k)$ is lcmg isomorphic to $B(X,v,k)$, so that $(\mu_{v_n}^{X_n})^{(k)} = (\mu_v^X)^{(k)}$. Because the measure $\mu_v^X$ is compactly supported, it is characterized by its moments, and hence the convergence of the moments of the measures implies the weak convergence of the measures in this case, by Theorem \ref{MomentsConvergence_CharacterizedBy_Moments_WeakConvergence}.\bigskip
\end{proof}

% \section{Cayley lcmgs}
% \label{Cayley_lcmg}

\section{Proof of Theorem \ref{PointwiseConvergence_SDF_ResiduallyFiniteGroup}}
\label{MainResult_Proof}

In this section, the concept of a Cayley lcmg is introduced; these lcmg's will serve as the analogue of Cayley graphs for arbitrary self-adjoint elements of $\C[G]$. To this end, consider a self-adjoint element $z \in \C[G]$. Then one may express $z$ in the form: \[z = \sum_{s \in S} \lambda_s s,\] where $S \subseteq G$ is a finite symmetric subset (i.e. $s \in S \implies s^{-1} \in S$), such that $\lambda_s \neq 0$, and all elements of $S$ are pairwise distinct. If necessary, proceed to extend $S$ to a symmetric set of generators of $G$. Therefore, without loss of generality, $S$ may be taken to be a finite symmetric set of generators of $G$, where the coefficients $\lambda_s$ are allowed to take value $0$ for some of the generators $s$. Note that because $z = z^\ast$, $\lambda_s = \overline{\lambda_{s^{-1}}}$ for all $s \in S$; it may also happen that the identity element $1 \in G$ belongs to $S$.\bigskip

Let $\Gamma_{G,S}$ be the Cayley graph of $G$ with respect to $S$. Edges will be made to correspond to right multiplication of generators. Modify $\Gamma_{G,S}$ as follows. If $1 \in S$, then at each vertex of $\Gamma_{G,S}$, there is a single self-loop. For each directed edge $x \stackrel{s}{\rightarrow} xs$ attach a label $\lambda_s$; for self-loops, attach the label $\lambda_1$. With basepoint taken as $1 \in G$, the resultant graph is a self-involutive lcmg, and will be called the \emph{Cayley lcmg of $w$ with respect to $S$}, denoted $\Gamma_{z,S}$.\bigskip

% Since $w \in \C[\bracket{S}] \cap \C[G]$ where $\bracket{S}$ is the subgroup generated by $S$, one has 2 right multiplication operators, depending on the choice of the group; $T_{w,S} = R_{w} \in \vonNeumannAlg{\bracket{S}}$, and $T_{w,G} = R_{i(w)} \in \vonNeumannAlg{G}$, where $i: \bracket{S} \hookrightarrow G$ is the inclusion. From Proposition \ref{GroupInduction_Properties}, $T_{w,G} = i_\ast T_{w,S}$.\bigskip

% Take the Cayley graph $\Gamma_{\bracket{S}}$ of $\bracket{S}$, the subgroup of $G$ generated by $S$. Edges will be made to correspond to \emph{right} multiplication of generators. Proceed to modify $\Gamma_{\bracket{S}}$ as follows. If $1 \in S$, then at each vertex of $\Gamma_{\bracket{S}}$, there is a single self-loop. For each directed edge $x \stackrel{s}{\rightarrow} xs$ attach a label $\lambda_s$; for self-loops, attach the label $\lambda_1$. With basepoint taken as $1 \in \bracket{S}$, it can be checked that the resultant graph, denoted $\Gamma_{T_{w,S}}$, is a self-involutive lcmg, and will be called the \emph{Cayley lcmg} of $T_{w,S}$.\bigskip

Further, suppose that $G$ is countable residually finite, as described in the Setting. Then $\pi_n(S) = \{sK_n : s \in S\}$ is a finite symmetric set of generators of $G/K_n$. Note that elements of $\pi_n(S)$ are taken without repetition, i.e. $\pi_n(S)$ is the collection of pairwise distinct cosets of $G/K_n$ obtained from the action of $\pi_n$ on $S$. Then $z_n = \pi_n(z)$ can be expressed in the form: \[z_n = \sum_{s_n \in \pi_n(S)} \lambda_{s_n} s_n.\] In a similar way, for each $n \in \N$, one obtains the Cayley lcmg $\Gamma_{z_n,\pi_n(S)}$ of $z_n$ with respect to $\pi_n(S)$, where the basepoint is taken to be the identity element $1_n \in G/K_n$.\bigskip

% The following is a basic algebraic result about countable residually finite groups $G$.
% 
% \begin{lemma}
% Let $G$ be a countable residually finite group as described in the Setting, and let $W \subseteq G$ be a finite subset. Then there exists $N \in \N$ such that, for all $n \geq N$, $\pi_n(W)$ consists of $|W|$ distinct cosets in $G/K_n$.
% \end{lemma}
% 
% \begin{proof}
% Suppose otherwise.
% \end{proof}

% Define $\bracket{S}_n = \{gK_n : g \in S\} = \bracket{\pi_n(S)}$, a subgroup of $G/K_n$. Let $w_n = \pi_n(w) \in \C[\bracket{S}_n] \cap \C[G/K_n]$. Then $w_n^\ast = w_n$, and $\pi_n(S)$ (where elements appear without repetition) is a finite symmetric subset. Express $w_n = \sum_{x_n \in \pi_n(S)} \lambda_{x_n} x_n$, where $\lambda_{x_n}$ may possibly be 0. Let $T_{w,S,n} := R_{w_n}: l_2\bracket{S}_n \rightarrow l_2\bracket{S}_n$ be the associated right multiplication operator. Take the Cayley graph $\Gamma_{\bracket{S}_n}$ with respect to the generating set $\pi_n(S)$, and from one also obtains the associated Cayley lcmg $\Gamma_{T_{w,S,n}}$, with basepoint the identity element $1_n \in \bracket{S}_n$.

\begin{propos}
\label{ResiduallyFinite_Cayleylcmg_Convergence}
The sequence of lcmg's $\{(\Gamma_{z_n,\pi_n(S)},1_n)\}_{n \in \N}$ converges to $\Gamma_{z,S}$.
\end{propos}

\begin{proof}
The aim is to apply Lemma \ref{lcmg_Convergence_EquivalentStatements}. Let $r > 0$. The property that $\bigcap_{n \in \N} K_n = \{1\}$ implies that there exists $N \in \N$ such that for all $n \geq N$, $V(B(\Gamma_{z,S},1,2r + 1)) \cap K_n = \{1\}$. This implies that: \begin{enumerate}
 \item The generating set $\pi_n(S)$ for $G/K_n$ contains $|S|$ pairwise distinct cosets (if $1 \in S$, then the identity coset $1K_n$ is also included in this collection of cosets).
 \item If $g \in K_n$ is such $g \neq 1$, then any word $w$ in the generators $S$ that represents $g$ must be of length greater than $2r + 1$. Equivalently, any directed path in $\Gamma_{w,S}$ starting at $1$ and ending at $g \in K_n$ must be of length greater than $2r + 1$.
\end{enumerate}

One must show that the map $$\alpha_n: V(B(\Gamma_{z,S},1,r)) \rightarrow V(B(\Gamma_{z_n,\pi_n(S)},1_n,r)),$$ given by $\alpha_n: x \mapsto \pi_n(x)$ is a lcmg isomorphism for $n \geq N$.\bigskip

$\alpha_n$ is a bijection. That $\alpha_n$ is surjective holds since $\pi_n$ is a projection. To prove injectivity, consider vertices $x,y \in B(\Gamma_{z,S},1,r)$. Then $x$ and $y$ can be expressed as words in the generators $S$, each of length at most $r$; each of these words would correspond to a path of length at most $r$ in $B(\Gamma_{z,S},1,r)$. If one has $\pi_n(x) = \pi_n(y)$, then $x^{-1}y \in K_n$ can be expressed as a word (in the generators $S$) of length at most $2r$. This means that there is a directed path in $\Gamma_{z,S}$ starting at $1$ and ending at an element of $K_n$ of length $2r$. By the choice of $N$, one must have $x^{-1}y = 1 \Longleftrightarrow x = y$.\bigskip

$\alpha_n$ is basepoint-preserving. This follows from $\pi_n(1) = 1_n$.\bigskip

$\alpha_n$ is edge-preserving. Note that in $\Gamma_{z,S}$ and $\Gamma_{z_n,S_n}$, two vertices $x, y \in G$ (respectively $xK_n$ and $yK_n$) are adjacent by means of a labelled directed edge $x \stackrel{\lambda_s}{\rightarrow} y$ (respectively $xK_n \stackrel{\lambda_s}{\rightarrow} yK_n$) if and only if $y = xs$ (respectively $yK_n = xK_n \cdot sK_n = xsK_n$). Since $\pi_n(xs) = \pi_n(x) \pi_n(s)$, an edge $x \stackrel{\lambda_s}{\rightarrow} xs$ in $B(\Gamma_{z,S},1,r))$ gives rise to an edge $xK_n \stackrel{\lambda_s}{\rightarrow} xsK_n$ in $B(\Gamma_{z_n,\pi_n(S)},1_n,r)$. Conversely, suppose vertices $x,y \in B(\Gamma_{z,S},1,r))$ are such that there is an edge $xK_n \stackrel{\lambda_s}{\rightarrow} yK_n$ in $\Gamma_{z_n,\pi_n(S)}$. Then $yK_n = xsK_n \Longleftrightarrow y^{-1}xs \in K_n$. Now express $x$ and $y$ as words of length at most $r$ in the generators $S$. Then $y^{-1}xs \in K_n$ can be expressed as a word of length at most $2r+1$ in the generators $S$. This would imply that there is a path of length at most $2r+1$ in $\Gamma_{z,S}$ starting at $1$ and ending at an element of $K_n$. By the choice of $N$, one must have that $y^{-1}xs = 1 \Longleftrightarrow y = xs$, which shows that there is a labelled directed edge $x \stackrel{\lambda_s}{\rightarrow} y$ in $\Gamma_{z,S}$.\bigskip

$\alpha_n$ is label-preserving. This is a direct consequence of the proof that $\alpha_n$ is edge-preserving; each labelled directed edge $x \stackrel{\lambda_s}{\rightarrow} xs$ in $B(\Gamma_{z,S},1,r))$ gives rise to a a labelled directed edge $xK_n \stackrel{\lambda_s}{\rightarrow} xsK_n$.\bigskip

Therefore, the proof is completed.
\end{proof}

% is the operator on  given by right multiplication of $w$ and by checking both operators on the orthonormal $\C$-bases $\bracket{S}$ and $\bracket{S}_n$, one has the following:

% \begin{propos}
% \label{GroupRingRightMult_CayleylcmgMarkovOper}
% $T = M_{\Gamma_{T}}$ as an operator on $l_2G$ and $T_{w,S,n} = M_{\Gamma_{T_{w,S,n}}}$ as an operator on $l_2\bracket{S}_n$.
% \end{propos}

One may now prove the main result:

\begin{proof}(of Theorem \ref{PointwiseConvergence_SDF_ResiduallyFiniteGroup})\bigskip

% The inclusions $i: \bracket{S} \rightarrow G$ and $i_n: \bracket{S}_n \rightarrow G/K_n$ induce group induction functors $i_\ast$ and $i_{n,\ast}$ respectively. By Proposition \ref{GroupInduction_Properties}, $i_{n,\ast} T_n$ is the right multiplication operator on $l_2(G/K_n)$ that is obtained by applying $\pi_n:G \rightarrow G/K_n$ to $i(w) \in \C[G]$. Therefore, by Lemma 2.5 of \cite{WLuck_Approximation_FDAnalogues}, there exists $K > 0$ such that: $$K \geq \sup\{\norm{i_\ast T} = \norm{T},\norm{i_{n,\ast}T_n} = \norm{T_n}: n \in \N\}$$

Let $w \in \C[G]$, and $w_n = \pi_n(w) \in \C[G/K_n]$. The spectral density function of $R_w$ is given by: \begin{align*}F(R_w)(\lambda) = \vontrace{G}{E^{(R_w)^\ast \circ R_w}_{\lambda^2}} = \vontrace{G}{E^{R_{w w^\ast}}_{\lambda^2}}\end{align*} The second equality follows from observations that: $$(R_w)^\ast = R_{w^\ast} \ \text{and} \ R_{w_1} \circ R_{w_2} = R_{w_2 w_1},$$ for $w_1, w_2 \in \C[G]$.\bigskip

Therefore, let $z = w w^\ast$ and $z_n = \pi_n(z) = w_n w_n^\ast$, both of which are self-adjoint. Consider the Markov-type operators $M = M_{\Gamma_{z,S},1}$ and $M_n = M_{\Gamma_{z_n,\pi_n(S)},1_n}$ associated to the Cayley lcmg's $\Gamma_{z,S}$ and $\Gamma_{z_n,\pi_n(S)}$ respectively. With this set-up, one checks that under the identifications $V(\Gamma_{z,S}) = G$ and $V(\Gamma_{z_n,\pi_n(S)}) = G/K_n$:

\begin{enumerate}
 \item $M = R_z$ and $M_n = R_{z_n}$.
 \item $\delta_1 \in l_2(V(\Gamma_{z,S}))$ and $\delta_{1_n} \in l_2(V(\Gamma_{z_n,\pi_n(S)}))$ are identified with $1 \in l_2G$ and $1_n \in l_2(G/K_n)$ respectively.
\end{enumerate}

By Proposition \ref{ResiduallyFinite_Cayleylcmg_Convergence}, the sequence $\{(\Gamma_{z_n,\pi_n(S)},1_n)\}_{n \in \N}$ converges to $\Gamma_{z,S}$. Theorem \ref{lcmgConvergence_Implies_KSMConvergence} then shows that the sequence of Kesten spectral measures $\{\mu^{\Gamma_{z_n,\pi_n(S)}}_{1_n}\}_{n \in \N}$ converges weakly to the Kesten spectral measure $\mu^{\Gamma_{z,S}}_1$. By Corollary \ref{WeakConvergenceMeasures_CDFPointwiseConvergence_Equiv}, this implies that for all $\lambda \geq 0$ such that $\mu^{\Gamma_{z,S}}_1(\lambda)$ is continuous, the sequence of cumulative distribution functions $\{\mu^{\Gamma_{z_n,\pi_n(S)}}_{1_n}(\lambda)\}$ converges pointwise to $\mu^{\Gamma_{z,S}}_1(\lambda)$.\bigskip

Therefore, what remains is to compute the $\mu^{\Gamma_{z_n,\pi_n(S)}}_{1_n}(\lambda)$ and $\mu^{\Gamma_{z,S}}_1(\lambda)$, and it will be shown that they coincide with $F(R_{w_n})(\lambda)$ and $F(R_w)(\lambda)$ respectively:\begin{align*}
\mu^{\Gamma_{z,S}}_1(\lambda) &= \inp{E^{M}_{[0,\lambda^2]}(\delta_1)}{\delta_1} = \vontrace{G}{E^M_{[0,\lambda^2]}}\\
&= \vontrace{G}{E^{R_{z}}_{[0,\lambda^2]}} = F(R_w)(\lambda)\\
\bigskip\\
\mu^{\Gamma_{z_n,\pi_n(S)}}_{1_n}(\lambda) &= \inp{E^{M_n}_{[0,\lambda^2]}(\delta_{1_n})}{\delta_{1_n}} = \vontrace{G/K_n}{E^{M_n}_{[0,\lambda^2]}}\\
& = \vontrace{G/K_n}{E^{R_{z_n}}_{[0,\lambda^2]}} = F(R_{w_n})(\lambda)\\
% \bigskip\\
% \implies F(R_w)(\lambda) = \vontrace{G}{E^{R_{z}}_{[0,\lambda^2]}} = \vontrace{G}{E^M_{[0,\lambda^2]}}\\
% \implies F(R_{w_n})(\lambda) = \vontrace{G/K_n}{E^{R_{z_n}}_{[0,\lambda^2]}} = \vontrace{G/K_n}{E^{M_n}_{[0,\lambda^2]}}
\end{align*}

Thus, the proof is completed.

% By Proposition \ref{GroupRingRightMult_CayleylcmgMarkovOper} and Proposition \ref{GroupInduction_Properties}, the Kesten spectral measures of the Cayley lcmgs $(\Gamma_{T},1)$ and $(\Gamma_{T_n},1)$ is given by:

% By Theorem \ref{lcmgConvergence_Implies_KSMConvergence}, the sequence of Kesten spectral measures $\{\mu^{\Gamma_{w_m,\pi_n(S)}}_{1_n}\}_{n \in \N}$ converges weakly to $\mu^{\Gamma_{w,S}}_1$, and by Corollary \ref{WeakConvergenceMeasures_CDFPointwiseConvergence_Equiv}, the proof is complete.
\end{proof}

\end{document}